\documentclass[11pt,letterpaper]{amsart}

\usepackage{xcolor}
\usepackage{amsmath}
\usepackage{amssymb}
\usepackage{enumitem}
\usepackage{tikz,float}
\usetikzlibrary{cd,patterns}

\usepackage[bb=ams,scr=euler]{mathalpha}

\usepackage{hyperref}

\usepackage{thmtools}
\declaretheoremstyle[headfont=\normalsize\normalfont\bfseries,notefont=\mdseries,
notebraces={(}{)},bodyfont=\normalfont\itshape,postheadspace=0.5em, spaceabove=3mm, spacebelow=2mm]{italstyle}
\declaretheorem[style=italstyle,name=Theorem,numberwithin=section]{theorem}
\declaretheorem[style=italstyle,name=Corollary,sibling=theorem]{corollary}
\declaretheorem[style=italstyle,name=Claim,sibling=theorem]{claim}
\declaretheorem[style=italstyle,name=Proposition,sibling=theorem]{prop}
\declaretheorem[style=italstyle,name=Lemma,sibling=theorem]{lemma}

\makeatletter
\newcommand{\abs}[1]{|#1|}
\newcommand{\bd}{\partial}
\newcommand{\C}{\mathbb{C}}
\renewcommand{\d}{\mathrm{d}}
\newcommand{\id}{\mathrm{id}}
\newcommand{\ip}[1]{\left\langle#1\right\rangle}
\newcommand{\pd}[2]{\frac{\partial #1}{\partial #2}}
\newcommand{\R}{\mathbb{R}}
\newcommand{\Z}{\mathbb{Z}}

\def\@secnumfont{\bfseries}
\renewcommand\section{\@startsection{section}{1}{0pt}{-3.5ex \@plus -1ex \@minus -.2ex}{2.3ex \@plus.2ex}{\centering\itshape}}
\newcommand{\set}[1]{\{#1\}}
\renewcommand{\subsection}{\@startsection{subsection}{2}%
  \z@{.5\linespacing\@plus.7\linespacing}{-.5em}%
  {\normalfont\itshape}}
\renewcommand{\paragraph}{\@startsection{paragraph}{4}%
  \z@{-.3em}\z@{\normalfont\itshape}}

\def\l@paragraph{\@tocline{4}{0pt}{1pc}{7pc}{}}
\makeatother

\title{The spectral diameter of a symplectic ellipsoid}
\author{Habib Alizadeh}
\author{Marcelo S. Atallah}
\author{Dylan Cant}

\begin{document}
\maketitle

\begin{abstract}
  The spectral diameter of a symplectic ball is shown to be equal to its capacity; this result upgrades the known bound by a factor of two and yields a simple formula for the spectral diameter of a symplectic ellipsoid. We also study the relationship between the spectral diameter and packings by two balls.
\end{abstract}

\section{Introduction}
\label{sec:introduction}

\subsection{Spectral diameter as a capacity}
\label{sec:diameter-as-a-capacity}

A well-known construction in Floer theory associates a \emph{spectral invariant} to a compactly supported Hamiltonian system $\varphi_{t}$ on a convex-at-infinity symplectic manifold $W$. The sum of the spectral invariants of $\varphi_{t}$ and its inverse is called the \emph{spectral norm} $\gamma(\varphi_{t})$. For an open set $U\subset W$ one can therefore consider the \emph{spectral diameter}:
\begin{equation*}
  \gamma(U)=\sup\set{\gamma(\varphi_{t}):\varphi_{t}\text{ is supported in }U}.
\end{equation*}
Such a quantity is a symplectic capacity for $U$ in the sense of, e.g., \cite{cieliebak-hofer-latschev-schlenk}, and has been considered in \cite{schwarz_spectral_invariants,frauenfelder-schlenk,mailhot-preprint}; for further discussion see \S\ref{sec:spectral-invariants}. Our main result is the exact formula for the spectral diameter of a symplectic ellipsoid in $W=\C^{n}$:
\begin{theorem}\label{theorem:main}
  The spectral diameter of the ellipsoid: $$E(a_{1},\dots,a_{n})=\textstyle\set{z:\sum \pi a_{i}^{-1}\abs{z_{i}}^{2}<1},$$ with  $a_{1}\le \dots\le a_{n}$, is equal to:
  \begin{equation}
    \gamma(E(a_{1},\dots,a_{n}))=\left\{
      \begin{aligned}
        &a_{n}&&\text{ if }a_{n}\in [a_{1},2a_{1}],\\
        &2a_{1}&&\text{ if }a_{n}\in [2a_{1},\infty);
      \end{aligned}
    \right.
  \end{equation}
  in particular, $\gamma(B(1))=1$ and $\gamma(Z(1))=2$.
\end{theorem}
To the authors' knowledge, the equality $\gamma(B(1))=1$ has so far not appeared in the literature. The inequality $\gamma(Z(1))\le 2$ follows from a displacement energy bound and has been observed before; see \S\ref{sec:outline-argument} and \S\ref{sec:displ-energy-bound}.

\subsection{Outline of argument}
\label{sec:outline-argument}

The argument proving Theorem \ref{theorem:main} is divided into three main steps:

\begin{enumerate}
\item \emph{The case of a cylinder}: $\gamma(E(1,\infty,\dots))\le 2$. This step is proved using a well-known upper bound on the spectral diameter in terms of the displacement energy; see \S\ref{sec:displ-energy-bound} and \S\ref{sec:cyl-upper-bound}.
\item \emph{The case of a long ellipsoid}: $\gamma(E(1,\dots,2))\ge 2$. This lower bound uses the standard moment map $\R^{2n}\to [0,\infty)^{n}$ and toric geometry. Briefly, one shows that $E(1,\dots,2)$ contains two balls of capacity $1$, and then explicitly constructs systems supported in these two balls to obtain the stated lower bound; see \S\ref{sec:toric-geom-pack} and \S\ref{sec:spectr-diam-pack}.
\item \emph{The case of a ball}: $\gamma(B(1))=1$. This step is the most delicate; the argument relies on the Hamiltonian circle action on $\C^{n}$ which rotates all the coordinates, and analyzing the effect on the action filtration of Floer homology groups. The proof is given in \S\ref{sec:spectr-diam-ball}.
\end{enumerate}

Assuming (1), (2), and (3), we prove the theorem. It is well-known that:
\begin{equation*}
  \gamma(\sqrt{a}U)=a\gamma(U)
\end{equation*}
for open sets $U\subset \R^{2n}$; see, e.g., \cite[\S2]{cieliebak-hofer-latschev-schlenk}. Therefore, using (1) and (2), if $a_{n}\in [2a_{1},\infty)$, we have:
\begin{equation*}
  \gamma(E(a_{1},\dots,a_{n}))=a_{1}\gamma(E(1,\dots,a_{n}/a_{1}))=2a_{1},
\end{equation*}
because $E(1,\dots,2)\subset E(1,\dots,a_{n}/a_{1})\subset E(1,\infty,\dots)$. On the other hand, if $a_{n}\in [a_{1},2a_{1}]$, then a similar argument yields:
\begin{equation*}
  \gamma(E(a_{1},\dots,a_{n}))=\frac{a_{n}}{2}\gamma(E(2a_{1}/a_{n},\dots,2))\ge \frac{a_{n}}{2}\gamma(E(1,\dots,2))=a_{n},
\end{equation*}
while (3) implies:
\begin{equation*}
  \gamma(E(a_{1},\dots,a_{n}))\le \gamma(E(a_{n},\dots,a_{n}))=a_{n},
\end{equation*}
so $\gamma(E(a_{1},\dots,a_{n}))=a_{n}$ for $a_{n}\in [a_{1},2a_{1}]$. This is what we wanted to show.

\subsection{The spectral diameter and packings by two balls}
\label{sec:spectr-diam-pack-1}

Part of the argument used in the proof of Theorem \ref{theorem:main} involves the additivity of the spectral diameter with respect to packings by two balls. We state this result as it is of independent interest:
\begin{theorem}\label{theorem:packings-by-two-balls}
  Let $U$ be an open set in an aspherical and convex-at-infinity symplectic manifold $W$. Suppose there exists a symplectic embedding:
  \begin{equation*}
    B(a_{1})\sqcup B(a_{2})\to U;
  \end{equation*}
  then $a_{1}+a_{2}\le \gamma(U)$.
\end{theorem}
Such additivity does not hold for packings by three or more balls; for example, the cylinder in $\C^{n}$ contains infinitely many disjoint balls with a given capacity but has a bounded spectral diameter.

Theorem \ref{theorem:packings-by-two-balls} also enables one to conclude:
\begin{corollary}\label{cor:spec_diam_pol_disks}
  The spectral diameter of a polydisk: $$P(a_{1},\dots,a_{n})=D(a_{1})\times \dots \times D(a_{n}),$$ with $a_{1}\le \dots \le a_{n}$, is equal to $2a_{1}$ if $n\ge 2$; if $n=1$ it is equal to $a_{1}$.
\end{corollary}
The spectral diameter obstructs symplectic embeddings of $P(a,a)$ into $B(a')$ unless $2a\le a'$; when $n=2$, this embedding obstruction recovers the result proved in \cite{ekeland-hofer-2}. In higher dimensions the spectral diameter obstruction is weaker than the one obtained in \cite{ekeland-hofer-2}.

\subsection{A two-ball capacity}
\label{sec:two-ball-capacity}

Consider the capacity:
\begin{equation*}
  c_{2B}(U):=\sup\set{a+b:\text{ there exists a symplectic } B(a)\sqcup B(b)\subset U}.
\end{equation*}
Our method proves that $c_{2B}(U)=\gamma(U)$ whenever $U$ is a symplectic ellipsoid or polydisk. This begs the question: {\itshape what is the largest class $\mathscr{C}$ of domains $U$ for which $c_{2B}(U)=\gamma(U)$?} The construction of \cite{hermann-starshaped-preprint-1998} shows that the class $\mathscr{C}$ does not contain certain starshaped domains; it is based on the fact that arbitrarily small neighborhoods of the torus $\bd D(r)\times \dots \times \bd D(r)$ have a large spectral diameter (proportional to $r$) but with a small $c_{2B}$ capacity.

The equality $c_{2B}(B(1))=\gamma(B(1))=1$ implies Gromov's two-ball theorem \cite[0.3.B]{gromov-inventiones-1985}. Gromov's method yields $c_{2B}(\mathbb{C}P^{n})=1$; this fact can also be proved using spectral diameter. Indeed, it follows from \S\ref{sec:digr-non-asph} that:
\begin{equation*}
  c_{2B}(\mathbb{C}P^{n})\le \gamma(\mathbb{C}P^{n}).
\end{equation*}
When combined with the result of \cite{entov-poltero-IMRN-2003} that $\gamma(\mathbb{C}P^{n})\le 1$ this upper bound establishes the two-ball theorem for $\mathbb{C}P^{n}$.

\subsection{On the spectral displacement energy}
\label{sec:spectr-displ-energy}

The \emph{spectral displacement energy} of a precompact open set $U\subset W$ is the value:
\begin{equation*}
  e_{\gamma}(U)=\inf\{\gamma(\psi_t):\psi_1(U)\cap U=\emptyset\};
\end{equation*}
such a quantity was considered in \cite[\S4]{viterbo92-GF}. Combining the displacement energy bound of \S\ref{sec:displ-energy-bound} and the existence of a certain compactly supported Hamiltonian system displacing $B(1)$ from itself it is possible to prove the following:
\begin{theorem}\label{theorem:displacement_energy}
The spectral displacement energy $e_{\gamma}(B(1))$ is equal to 1.
\end{theorem}
The construction of the system is recalled in \S\ref{sec:cyl-upper-bound}. This equality shows one does not obtain the spectral diameter $\gamma(B(1))=1$ directly from the displacement energy bound from \S\ref{sec:displ-energy-bound}. Indeed, the displacement energy bound yields:
\begin{equation*}
  \gamma(B(1))\leq2e_{\gamma}(B(1))=2,
\end{equation*}
which is suboptimal in view of Theorem \ref{theorem:main}.

\subsection{Spectral diameter of special balls in projective space}

Consider $\mathbb{C}P^{n}$ with the Fubini-Study symplectic form, normalized so that the class of $\mathbb{C}P^{1}$ has symplectic area $1$. In this case, it is well-known that the symplectic structure on $\mathbb{C}P^{n}$ is determined as the symplectic reduction of $\bd B(1)\subset \C^{n+1}$.

Let $a\in (0,1)$. The function $S_{a}=\pi \abs{z_{0}}^{2}-a$ is well-defined on $\mathbb{C}P^{n}$ via the quotient map $\bd B(1)\to \mathbb{C}P^{n}$, and generates a 1-periodic Hamiltonian circle action on $\mathbb{C}P^{n}$.

\begin{figure}[h]
  \centering
  \begin{tikzpicture}[scale=1.9]
    \draw (0,0)--(1,0);
    \draw (0,-0.3)node[left]{$-a$} --+(0,1) node[left]{$1-a$};
    \draw[line width=0.6pt] (0,-0.3)--(1,0.7);
    \draw[dotted,line width=0.7pt] (0,-0.3)--+(0,-0.2)node[below]{$\mathbb{C}P^{n-1}$} (1,0.7)--+(0,-1.2)node[below]{$\mathbb{C}P^{0}$};
  \end{tikzpicture}
  \caption{Schematic illustration of $S_{a}$, graphed as a function of $\pi\abs{z_{0}}^{2}$. The point $\mathbb{C}P^{0}$ is the maximum of $S_{a}$ and corresponds to the line $z_{1}=\dots=z_{n}=0$. The divisor $\mathbb{C}P^{n-1}$ is the Morse-Bott minimum of $S_{a}$ and represents all lines passing through the hyperplane $z_{0}=0$.}
  \label{fig:schematic-Sa}
\end{figure}
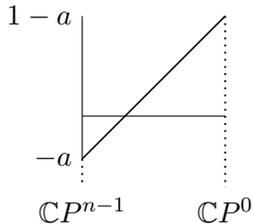

The set $\set{S_{a}\ge 0}$ is symplectomorphic to a ball of capacity $1-a$, a fact whose verification is left to the reader. Let us call a ball in $\mathbb{C}P^{n}$ \emph{special} if it is obtained by applying a Hamiltonian diffeomorphism to $\set{S_{a}\ge 0}$, for some $a\in (0,1)$. Beyond dimensions $n=1,2$, it does not seem to be known whether the image of every embedding of a closed ball into $\mathbb{C}P^{n}$ is special.

Using methods similar to our proof of Theorem \ref{theorem:main}, we prove:
\begin{theorem}\label{theorem:A}
  The spectral diameter of a special ball is equal to its capacity.
\end{theorem}
Here the spectral diameter is computed using the spectral invariants within $\mathbb{C}P^{n}$ and uses the coefficient field $\Z/2$.

A special ball can be parametrized by an embedding $i:B(a)\to \mathbb{C}P^{n}$. Thus any compactly supported Hamiltonian system $\psi_{t}$ on $B(a)$ can be ``pushed forward'' to $\mathbb{C}P^{n}$ by the formula $i\psi_{t}i^{-1}$, extended to the complement of the ball as the identity system. It is natural to wonder whether:
\begin{equation}\label{eq:equality-of-spectral-invariants}
  \gamma_{\C^{n}}(\psi_{t})=\gamma_{\mathbb{C}P^{n}}(i\psi_{t}i^{-1})
\end{equation}
holds for every system $\psi_{t}$. As we show in \S\ref{sec:comp-with-bound}, the equality \eqref{eq:equality-of-spectral-invariants} fails in general (our example requires $a$ to be close to $1$). This is noteworthy as it shows we cannot simply appeal to the known bound on the spectral diameter of $\mathbb{C}P^{n}$ from \cite{entov-poltero-IMRN-2003} to deduce $\gamma_{\C^{n}}(\psi_{t})\le 1$ for all $\psi_{t}$ supported in $B(1)$.

Interestingly enough, if \eqref{eq:equality-of-spectral-invariants} holds for even an arbitrarily small special ball, then Theorem \ref{theorem:A} can be used to recover Theorem \ref{theorem:main} for balls. Establishing sufficient conditions to ensure \eqref{eq:equality-of-spectral-invariants} in the presence of symplectic spheres seems to be a non-trivial task, even for small balls, and we save further study of \eqref{eq:equality-of-spectral-invariants} for future research.

\subsection{Acknowledgements}
\label{sec:acknowledgements}

The authors benefitted from valuable discussions with E.~Shelukhin and O.~Cornea. This research was undertaken at Universit\'e de Montr\'eal with funding from the Fondation Courtois, the ISM, the FRQNT, and the Fondation J.~Armand Bombardier.

\section{Proofs}
\label{sec:proofs}

\subsection{Floer homology in convex-at-infinity manifolds}
\label{sec:floer-cohom-convex}

This section is concerned with a recollection of various Floer theoretic objects used in this paper.

\subsubsection{Cappings}
\label{sec:cappings}

For each fixed point $x$ of the time-one map $\psi_{1}$, a \emph{representative capping} is a smooth map $u:[-1,1]\times \R/\Z\to W$ so that $u(-1,t)=x$ and $u(1,t)=\psi_{t}(x)$. Representative cappings are considered up to equivalence: the difference of two representative cappings forms a sphere and if this sphere has zero symplectic area then the representatives are declared to be equivalent. An equivalence class of representatives will be referred to as a \emph{capping}. Capped orbits are denoted as pairs $(x,u)$.

Requiring that $u(-1,t)=x$ has the following advantage:
\begin{lemma}
  If $u$ is a capping of $x$ then $\bar{u}(s,t)=\psi_{t}^{-1}(u(-s,t))$ is a capping of $x$ for the system $\psi_{t}^{-1}$, and the action of $\bar{u}$ is minus the action of $u$.\hfill$\square$
\end{lemma}

\subsubsection{Action}
\label{sec:action}

To each capping one can associate an action:
\begin{equation*}
  \mathscr{A}(\psi_{t};x,u)=\int H_{t}(\psi_{t}(x))\d t-\int u^{*}\omega,
\end{equation*}
where $H_{t}$ is the normalized generator for a contact-at-infinity Hamiltonian system $\psi_{t}$. For simplicity we suppose that $W$ is connected, and we consider two classes of normalization in this paper:
\begin{enumerate}
\item If $W$ is open, and $Y_{0}$ is a chosen connected component of the ideal boundary of $W$, then a Hamiltonian function $H_{t}$ is normalized if it is one-homogeneous in the non-compact end corresponding to $Y_{0}$; see \cite{alizadeh-atallah-cant} for the definitions of \emph{ideal boundary} and \emph{one-homogeneous}.
\item If $W$ is closed, then a Hamiltonian function $H_{t}$ is normalized if the integral of $H_{t}\omega^{n}$ over $W$ vanishes for each $t$.
\end{enumerate}

In (1) different choices of $Y_{0}$ give different normalizations. The crucial properties are that the set of normalized Hamiltonians is a vector subspace and a constant normalized function is zero.

\subsubsection{The Floer homology vector space}
\label{sec:floer-chain-complex}

Let $\psi_{t}$ be a contact-at-infinity system and suppose that $\psi_{1}$ has non-degenerate fixed points.

Define $\mathrm{CF}(\psi_{t})$ to be the $\Z/2$-vector space of semi-infinite sums generated by capped orbits $(x,u)$ of $\psi_{t}$, requiring that $\mathscr{A}(\psi_{t};x,u)\ge L$ holds for only finitely many terms in the sum, for each $L$.

\subsubsection{The Floer homology differential}
\label{sec:floer-hom-diff}

The Floer differential depends on a choice of almost complex structure $J_{t}$, although different choices give isomorphic chain complexes. It is defined as usual in fixed point Floer homology; see, e.g., \cite{dostoglou_salamon,seidel-eq-pop}. The relevant moduli space $\mathscr{M}(\psi_{t},J_{t})$ is the space of twisted holomorphic curves:
\begin{equation*}
  \left\{
    \begin{aligned}
      &w:\C\to W,\\
      &\bd_{s}w+J_{t}(w)\bd_{t}w=0,\\
      &\psi_{1}(w(s,t+1))=w(s,t).
    \end{aligned}
  \right.
\end{equation*}
In order for the cylinder $u(s,t)=\psi_{t}(w(s,t))$ to solve a smooth PDE, we require that $J_{t}$ is $\psi_{1}$-\emph{twisted-periodic}, i.e., $J_{t+1}(w)=\d\psi_{1}^{-1}J_{t}(\psi_{1}(w))\d\psi_{1}.$

By counting rigid-up-to-translation elements in $\mathscr{M}(\psi_{t},J_{t})$, with $w(-\infty)$ considered as input and $w(+\infty)$ considered as output one obtains a map:
\begin{equation*}
  d_{\psi_{t},J_{t}}:\mathrm{CF}(\psi_{t})\to \mathrm{CF}(\psi_{t}).
\end{equation*}
One uses the cylinder $u$ to determine the capping of the output in terms of the capping of the input. With these homological conventions, the Floer differential decreases action. The homology of $(\mathrm{CF}(\psi_{t}),d_{\psi_{t},J_{t}})$ is denoted\footnote{To be pedantic, $\mathrm{HF}(\psi_{t})$ should be defined as a limit of the homologies of $(\mathrm{CF}(\psi_{t}),d_{\psi_{t},J_{t}})$ as $J_{t}$ varies over all admissible complex structures.} by $\mathrm{HF}(\psi_{t})$.

\paragraph{}\label{sec:inverse-identification}

It follows that $J_{-t}$ is $\psi_{1}^{-1}$ twisted periodic which yields an inversion identification $\iota:\mathscr{M}(\psi_{t},J_{t})\to \mathscr{M}(\psi_{t}^{-1},J_{-t})$ given by $\iota(w)(s,t)=w(-s,-t)$.

\subsubsection{Reeb flows}
\label{sec:reeb-flows}

This section is only relevant when $W$ is open. We describe the set-up on $\C^{n}$, although everything holds verbatim on a general convex-at-infinity manifold $W$ if one replaces $\pi \abs{z}^{2}$ by a suitable function $r$.

Let $R_{\delta,s}$ be the Hamiltonian flow generated by: $$\mu_{\delta}(\pi\abs{z}^{2}-1)+1$$ where $\mu_{\delta}$ is a convex cut-off function so that:
\begin{enumerate}
\item $\mu_{\delta}(x)$ is the constant $\delta/2$ for $x\le 0$,
\item $\mu_{\delta}(x)=x$ for $x\ge \delta$,
\item $\mu_{\delta}'(x)>0$ for $x>0$.
\end{enumerate}
It is important that $R_{\delta,s}$ agrees with the flow of $\pi\abs{z}^{2}$ for $\pi\abs{z}^{2}>1+\delta$, i.e., the ideal restriction of $R_{\delta,s}$ is the standard one-periodic Reeb flow.

\subsubsection{Spectral invariants}
\label{sec:spectral-invariants}

For any class $a\in \mathrm{HF}(\psi_{t})$, let:
\begin{equation*}
  c(\psi_{t};a):=\textstyle\inf\set{\sup_{i}\mathscr{A}(\psi_{t};x_{i},u_{i}):\text{the cycle }\sum_{i} (x_{i},u_{i})\text{ represents }a};
\end{equation*}
loosely speaking, $c(\psi_{t};a)$ is a homological min-max over all representative cycles; see, e.g., \cite{bialy-polterovich-duke-1994,schwarz_spectral_invariants,oh-2005-duke,frauenfelder-ginzburg-schlenk,frauenfelder-schlenk,usher08-spect-floer}.

When $\psi_{t}$ is convex-at-infinity and its ideal restriction is a negative Reeb flow there is a distinguished class $1\in \mathrm{HF}(\psi_{t})$ which plays the role of the unit element for the pair-of-pants product; see \S\ref{sec:sub-addit-spectr}. This is initially defined for non-degenerate systems. One extends the definition of $\mathrm{HF}(\psi_{t})$ and the unit element to all systems whose ideal restriction is a non-positive Reeb flow as an inverse limit over continuation maps. The spectral invariant $c(\psi_{t};1)$ extends to the limit; we refer the reader to \cite{alizadeh-atallah-cant} for details on continuation maps and this extension.

\subsubsection{The spectral norm}
\label{sec:spectral-norm}

The \emph{spectral norm} of a compactly supported system $\psi_{t}$ is defined by:
\begin{equation}\label{eq:spectral-norm}
  \gamma(\psi_{t}):=c(\psi_{t};1)+c(\psi_{t}^{-1};1)=\lim_{s\to 0}c(R_{-st}\psi_{t};1)+c(R_{-st}\psi_{t}^{-1};1),
\end{equation}
where $R_{s}$ is one of the Reeb flows constructed in \S\ref{sec:reeb-flows} (the parameters going into the definition of $R$ do not matter in the limit $s\to 0$). It is crucial that $\psi_{t}$ is compactly supported, as the ideal restrictions of $\psi_{t}$ and $\psi_{t}^{-1}$ are then both non-positive.

A similar definition of a spectral norm in the case of $\C^{n}$ appears in \cite{viterbo92-GF} using the framework of generating functions. The paper which introduced the spectral norm using the framework of Floer theory in aspherical manifolds is \cite{schwarz_spectral_invariants}, and the extension to convex-at-infinity $W$ is due to \cite{frauenfelder-schlenk}.

\subsubsection{Spectral diameter}
\label{sec:spectral-diameter}

As explained in the introduction, one obtains a \emph{spectral diameter}\footnote{If one defines the distance $d(\varphi_{t},\psi_{t})=\gamma(\varphi_{t}^{-1}\psi_{t})$ then $\gamma(U)$ is the diameter of the set of all systems supported in $U$.} $\gamma(U)$, for any open set $U\subset W$, as the supremum of $\gamma(\psi_{t})$ over systems $\psi_{t}$ with compact support in $U$. The spectral diameter as a capacity is considered in \cite{schwarz_spectral_invariants,frauenfelder-schlenk}; an earlier capacity appears in \cite{viterbo92-GF}, namely the \emph{spectral capacity}:
\begin{equation*}
  c(U):=\sup\set{c(H_{t};1):\text{$H_{t}$ has compact support in $U$}},
\end{equation*}
which is known to be a normalized capacity on $\C^{n}$. See \S\ref{sec:non-norm-hamilt} for our conventions used to define $c(H_{t};1)$ on closed manifolds.

\subsubsection{Sub-additivity for spectral invariants and pair-of-pants product}
\label{sec:sub-addit-spectr}

Suppose that $\varphi_{t},\psi_{t}$ are two contact-at-infinity systems. The pair-of-pants product operation is a map:
\begin{equation*}
  \ast_{\mathrm{PP}}:\mathrm{HF}(\varphi_{t})\otimes \mathrm{HF}(\psi_{t})\to \mathrm{HF}(\varphi_{t}\psi_{t}),
\end{equation*}
defined by counting solutions to Floer's equation over a pair-of-pants surface, as in, e.g., \cite{schwarz-thesis,schwarz_spectral_invariants,seidel-eq-pop,kislev_shelukhin,alizadeh-atallah-cant}. The precise structure of Floer's equation involves the choice of a Hamiltonian connection over the pair-of-pants, and using a connection with zero curvature produces sharp energy estimates in terms of the actions of asymptotics. These energy estimates imply the sub-additivity of the spectral invariants:
\begin{equation*}
  c(\varphi_{t}\psi_{t};a\ast_{\mathrm{PP}}b)\le c(\varphi_{t};a)+c(\psi_{t};b).
\end{equation*}
For further discussion we refer the reader to \cite[\S2.4]{alizadeh-atallah-cant} which describes in detail the pair-of-pants product in convex-at-infinity symplectic manifolds.

It is important to note that, if $\varphi_{t}=R_{-\epsilon t}$ is a Reeb flow with a small negative speed, and $a=1$ is the unit element, then $1\ast_{\mathrm{PP}}b$ is the image of $b$ under the continuation map $\mathrm{HF}(\psi_{t})\to \mathrm{HF}(\varphi_{t}\psi_{t})$; see \cite{kislev_shelukhin}.

\subsubsection{Non-normalized Hamiltonians}
\label{sec:non-norm-hamilt}

It is sometimes convenient to generalize the definition to allow non-normalized Hamiltonian functions via the rule:
\begin{equation*}
  c(H_{t}+f(t);a)=c(\psi_{t};a)+\int_{0}^{1}f(t)dt,
\end{equation*}
where $H_{t}$ is the normalized generator for $\psi_{t}$, and $f(t)$ is a time-dependent shift. If we use the symbol $c(\psi_{t};a)$, then we require using the normalized generator. On the other hand, if we use the symbol $c(H_{t};a)$, then we allow $H_{t}$ to be non-normalized.

\subsection{The displacement energy bound}
\label{sec:displ-energy-bound}

In this section we recall the displacement energy bound on the spectral norm:
\begin{prop}\label{prop:displ-energy-bound}
  Let $W$ be a rational, semipositive, and convex-at-infinity symplectic manifold, and let $\psi_{t},\varphi_{t}$ be two compactly supported Hamiltonian systems such that $\psi_1$ displaces the support of $\varphi_t$. Then $\gamma(\varphi_t)\leq2\gamma(\psi_t)$. Moreover, if $W$ is open, then $c(\varphi_{t};1)\le \gamma(\psi_{t})$.
\end{prop}
\begin{proof}
  See \cite{ginzburg-2005-weinstein} which proves the result in the both the open and the closed case, assuming $W$ is aspherical; see also \cite{viterbo92-GF,hofer-zehnder-book-1994,schwarz_spectral_invariants,oh-2005-duke,usher-CCM-2010}. The argument extends easily from aspherical to rational and semipositive, all that is required is that spectral invariants are well-defined and are sub-additive (we use semipositivity to ensure their well-definedness) and valued in the action spectrum (this is why we assume rationality).
\end{proof}

\subsection{An upper bound to the capacity of a cylinder}
\label{sec:cyl-upper-bound}
In this section we bound the spectral displacement energy of the cylinder following closely the ideas in \cite[\S 2.4]{polterovich-book}.

A disc $B(1)\subset\C$ is symplectomorphic to a square with sides of length $1$; in particular, $Z(1)$ is symplectomorphic to $R=(0,1)^2\times\C^{n-1}\subset\C^n$. Therefore, it is enough to show that $\gamma(R)\leq2$. Consider the Hamiltonian system:\[\psi_t(x_1,y_1,\dots,x_n,y_n)=(x_1,y_1+t,\dots,x_n,y_n),\] generated by $H(x,y)=x_1$. Here, we identify $z_j=x_j+iy_j$ for all $j$. Then: \[\psi_1(R)\cap R=\emptyset,\] and the displacement occurs in $(0,1)\times(0,2)\times\C^{n-1}\subset\C^n$. Of course $\psi_t$ is neither compactly supported nor normalized.

Let $\varphi_t$ be a Hamiltonian system supported inside $(0,1)^2\times K\subset R$, for a compact set $K\subset\C^{n-1}$. Consider a Hamiltonian system $\psi_t'$ generated by a function $H'$ obtained by cutting-off $H$ outside an arbitrarily small neighbourhood of $(0,1)\times(0,2)\times K$; we can ensure that $$\max H'-\min H'\le 1+\delta$$ for some small $\delta$. This construction yields a system whose time-one map $\psi_1'$ displaces the support of $\varphi_t$. Moreover, by \S\ref{sec:displ-energy-bound}:
\begin{equation*}
  \gamma(\varphi_t)\leq 2\gamma(\psi_t')\leq 2(\max H'-\min H')=2+2\delta.
\end{equation*}
Here we use the well-known estimate that the spectral norm is less than the Hofer norm; see, e.g., \cite{kislev_shelukhin}. Since $\varphi_{t}$ is arbitrary and $\delta$ can be chosen arbitrarily small, we conclude the desired result $\gamma(Z(1))=\gamma(R)\leq 2$.

\subsection{Displacement energy of a ball}
\label{sec:disp-energy-ball}

In this section we show that the spectral displacement energy $e_{\gamma}(B(1))$ of $B(1)$ is equal to $1$; this is the statement of Theorem \ref{theorem:displacement_energy}.

We begin by showing $e_{\gamma}(B(1))\geq1$. Let $\psi_{t}$ be a Hamiltonian system such that $\psi_1$ displaces $B(1)$ and, hence, the support of any Hamiltonian system $\varphi_t$ supported therein. Proposition \ref{prop:displ-energy-bound} implies $c(\varphi_t;1)\leq\gamma(\psi_t)$.

It is known that one can find a bump function supported in $B(1)$ generating a Hamiltonian system $\varphi_t$ whose spectral invariant $c(\varphi_t;1)$ is arbitrarily close to $1$; indeed, this follows from the construction in \S\ref{sec:spectr-diam-pack}. In particular, one concludes that $\gamma(\psi_t)\geq 1$; this proves the desired inequality.

We now show that $e_{\gamma}(B(1))\leq1$; to that effect, it is enough to find a Hamiltonian system $\psi'_t$ whose time-1 map displaces $B(1)$ and $\gamma(\psi'_t)\leq1$. The system constructed in \S\ref{sec:cyl-upper-bound} satisfies the desired properties. This concludes the proof of Theorem \ref{theorem:displacement_energy}.

\subsection{Ball packings of toric domains}
\label{sec:toric-geom-pack}

This section explains the toric approach to ball packings following \cite{schlenk-05}. This is used to show that:
\begin{enumerate}
\item the ellipsoid $E(a_{1}, \dots, a_{n})$ with $a_{n} \geq 2a_{1}$, and
\item the polydisk $P(a_{1},\dots,a_{n})$, 
\end{enumerate}
each contain two disjoint balls whose capacities are arbitrarily close to $a_{1}$. Then Theorem \ref{theorem:packings-by-two-balls} bounds their spectral diameter from below by $2a_{1}$. On the other hand, the displacement energy bound of \S\ref{sec:displ-energy-bound} and the construction in \S\ref{sec:cyl-upper-bound} bounds their spectral diameter from above by $2a_{1}$. Such arguments played a role in the proof of Theorem \ref{theorem:main} and the results in \S\ref{sec:two-ball-capacity}.

Define the moment map $\mu: \mathbb{C}^{n}\rightarrow \mathbb{R}^{n}_{\geq 0}$ by $\mu(z)=(\pi |z_{1}|^{2}, \dots, \pi |z_{n}|^{2})$. Given a domain $D\subset \C^{n}$, its image $\mu(D)$ will be referred to as its \emph{toric image}. It is easy to check that the toric image of an ellipsoid is a simplex while the toric image of a polydisk is a rectangular parallelepiped. A domain $D$ is called \emph{toric} provided $D=\mu^{-1}(\mu(D))$; a moment's thought will reveal the both ellipsoids and polydisks are toric domains.

The standard technique used to construct a ball packing of a toric domain is to decompose its toric image into simplices, each of which are supposed to represent embedded symplectic balls. The key is the following symplectomorphism $\Phi$ between $\mathbb{R}^{n}_{>0} \times T^{n}$ and $\mathbb{C}^{n} \setminus \{z_{1}\cdots z_{n} = 0\}$: $$(a_{1}, \dots, a_{n}, \theta_{1}, \dots, \theta_{n}) \mapsto \frac{1}{\sqrt{\pi}}(\sqrt{a_{1}} e^{2\pi i  \theta_{1}}, \dots, \sqrt{a_{n}} e^{2\pi i \theta_{n}}),$$ where $(a_{1}, \dots, a_{n})$ are the coordinates on $\mathbb{R}^{n}_{> 0}$, $(\theta_{1}, \dots, \theta_{n})$ are the coordinates on $T^{n} = \mathbb{R}^{n}/ \mathbb{Z}^{n}$; the symplectic form on the domain is $\sum \d a_{i} \wedge \d\theta_{i}$.

It is a rather deep fact that: 
\begin{lemma}\label{lemma:schlenk}
  If $\Delta(a)\subset \R^{n}_{>0}$ is the open simplex consisting of open convex combinations of $0,ae_{1},\dots,ae_{n}$, then the Gromov width of $\mu^{-1}(\Delta(a))$ is $a$.
\end{lemma}
\begin{proof}
  For the proof see, e.g., \cite[\S3.1]{schlenk-05}. This is not immediate; indeed: $$\mu^{-1}(\Delta(a))=B(a)\setminus \set{z_{1}\dots z_{n}=0},$$ so any embedding of a ball must miss the removed parts. However, one can still embed balls with capacity arbitrarily close to $a$. 
\end{proof}

Introduce the group $G=\mathrm{SL}_{n}(\Z)\ltimes \mathbb{R}^{n}$ of special affine transformations. Then $G$ acts on $\R^{n}$ in a natural way, and the action extends to an action on $\R^{n}\times T^{n}$ by canonical transformations (in particular, the action is via symplectomorphisms). As a consequence, we conclude the following corollary of Lemma \ref{lemma:schlenk}: {\itshape if the toric image of a toric domain $\Omega$ contains a disjoint union $\Delta(a)\sqcup g(\Delta(a))$ where $g\in G$, then $\Omega$ contains two disjoint symplectic balls whose capacities are each arbitrarily close to $a$.}

Thus to prove that the ellipsoid $E(a_{1},\dots,a_{n})$ with $a_{n}\ge 2a_{1}$, and the polydisk $P(a_{1},\dots,a_{n})$ each contain two disjoint balls of capacity arbitrarily close to $a_{1}$, it suffices to cut its toric image into subsimplices; this can be done and is shown in Figure \ref{fig:affine_trans_ba}.

\begin{figure}[h]
  \centering
  \begin{tikzpicture}
    \fill[pattern=north east lines,pattern color=black!50!white] (0,0) -- (0,1) -- (2.5,0) -- (0,0);
    \fill[pattern=north east lines,pattern color=red!50!white] (0,0) -- (0,1) -- (1,0) -- (0,0);
    \fill[pattern=north east lines,pattern color=blue!50!white] (1,0) -- (0,1) -- (2,0) -- (1,0);
    \draw[<->] (3,0) -- (0,0) -- (0,2);
    \draw[thick,every node/.style={above,shift={(0,-.55)}}] (0,1) -- (1,0)node{$a_{1}$} (0,1)--(0,0)--(2,0)node {$2a_{1}$}--cycle (2.5,0)node{$a_{2}$};
    \draw[dashed] (0,1)node[left]{$a_{1}$} -- (2.5,0);

    \begin{scope}[shift={(5,0)}]
      \draw[dashed, ->] (0,0) -- (3, 1.5);
      \draw[<->] (3, 0) --(0,0) -- (0,2);

      \fill[pattern=north east lines,pattern color=red!50!white] (1,0) -- (0.8,0.4) -- (0,1) -- (0,0)--cycle;
      \fill[pattern=north east lines,pattern color=blue!50!white] (2.4,0.7) -- (1.4,1.7) -- (2.4,1.7) -- (2.4, 0.7);

      \draw[dashed] (1,0) -- (0.8,0.4)--(0,1);
      \draw[dashed] (1.4, 1.7) -- (2.4,0.7);
      \draw[dashed] (1.4,1.7) -- (1.4, 0.7) -- (2.4, 0.7);

      \draw[thick] (0,1) -- (1,0) --(0,0)--cycle;
      \draw[thick] (2.4,0.7) -- (2, 1.5) -- (1.4, 1.7);
      \draw[thick] (2.4,0.7)--(1,0)node[below]{$a_{1}$}--(1,1) (2.4,0.7) -- (2.4, 1.7)--(1.4, 1.7)--(0,1)node[left]{$a_{1}$}--(1,1)--(2.4,1.7);
    \end{scope}
  \end{tikzpicture}
  \caption{Decomposing the toric image of the ellipsoid (shown in dimension $2$) and the polydisk (shown in dimension $n=3$) into standard simplices. In the picture on the right we have set $a_{1} = a_{2}$.}
  \label{fig:affine_trans_ba}
\end{figure}
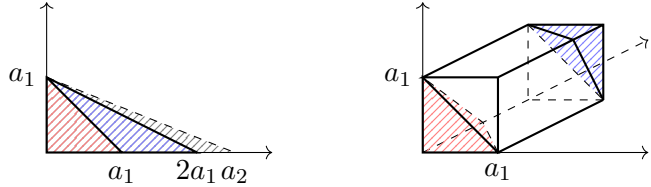

\subsection{Spectral diameter and packings of two balls}
\label{sec:spectr-diam-pack}

We present the proof of Theorem \ref{theorem:packings-by-two-balls} which provides a lower bound on the spectral diameter of domains in aspherical manifolds based on their packings by two balls. As in the statement of the theorem, let $U \subset W$ be an open domain in $W$, and let: $$B(a) \sqcup B(b) \to U$$ be an embedding of a disjoint union of Darboux balls. We will show that $\gamma(U) \geq a + b$ by explicitly constructing systems supported in these two balls. Other works have estimate spectral norms of systems supported in a disjoint union of balls, e.g., \cite{seyfaddini-spectral-killers-15,HLS-16,ishikawa-j-topol-anal-2016,tanny22,ganor-tanny-barricades}; these results are much more general and apply to any systems. Our result only requires a lower bound and can be deduced by elementary means.

The reason for assuming that $W$ is aspherical is so that the action is associated directly to the orbits (there is a single capping) and the spectral invariants lie in a compact nowhere dense action spectrum.

Consider the following Hamiltonian function:
$$H_{a,\eta,\delta}(z)=\mu_{\delta}(\eta(a-\pi \abs{z}^{2}))-\delta/2$$
where $0 < \eta < 1$ is any number, $\mu_{\delta}$ is the cut-off function in \S\ref{sec:reeb-flows}, and $\delta$ is much smaller than $\eta a$. It is important to note that this system has:
\begin{enumerate}
\item a $1$-periodic orbit at $z=0$ whose action is $\eta a-\delta/2$,
\item a family of $1$-periodic orbits when $\pi\abs{z}^{2}\ge a$ whose actions are $0$.
\end{enumerate}

Let $K_{a,b,\eta,\mu}$ be the system on $U$ obtained by implanting $H_{a,\eta,\delta}$ in the image of the ball $B(a)$ and the system $-H_{b,\mu,\delta}$ on the image of the ball $B(b)$.

Let $\psi_{t}$ be the generated system. One concludes:
\begin{enumerate}[label=(\alph*)]
\item the action spectrum of $\psi_{t}$ is the set $\{0, \eta a-\delta/2, \delta/2-\mu b\}+C$, and,
\item $\{0, \delta/2-\eta a, \mu b-\delta/2\}-C$ is the spectrum of $\psi_{t}^{-1}$.
\end{enumerate}
Here $C$ is a constant shift needed to normalize $K$, and is only necessary when $W$ is closed. Hence, the possible values for the spectral norm of $\psi_{t}$ are from the following set: $$\{\eta a-\delta/2, \mu b-\delta/2, \eta a + \mu b-\delta\}.$$ We now argue by cases. The key idea is to exploit continuity and spectrality of the spectral invariants, and non-degeneracy of the spectral norm.

In the first case we have $c(\psi_{t}) = \eta a-\delta/2+C$ and $c(\psi_{t}^{-1})=-C$; by taking $\eta \to 0$ and using continuity of the spectral invariants we obtain a system that is not the identity and has zero spectral norm which is a contradiction as the spectral norm is non-degenerate.

The second case, when $c(\psi_{t}) = C$ and $c(\psi_{t}^{-1})=\mu b-\delta/2-C$, is similarly ruled out. Hence the only possibility is that $\gamma(\psi_{t}) = \eta a + \mu b-\delta$. By taking $\delta$ to $0$ and $\eta,\mu$ to $1$ we conclude $\gamma(U) \geq a + b$, as desired.

\subsubsection{Beyond the aspherical case}
\label{sec:digr-non-asph}

The analogue of Theorem \ref{theorem:packings-by-two-balls} in the non-aspherical case is more subtle, and the argument in \S\ref{sec:spectr-diam-pack} requires some modification because the action functional becomes multivalued. In this section we refine the argument by taking into account the indices of the orbits.
\begin{prop}\label{prop:minimal-chern}
  Suppose that $(M^{2n},\omega)$ is a closed symplectic manifold and: $$\omega(u)\ne 0\implies \abs{c_{1}(u)}\ge n+1$$ for all $u\in \pi_{2}(M)$. If there is a symplectic embedding $B(a)\sqcup B(b)\to U$, where $U\subset M$ is an open set, then $\gamma(U)$ is at least $a+b$.
\end{prop}
\begin{proof}
  Consider the Hamiltonian $K_{a,b,\eta,\mu}$ as above, and take a small Morse perturbation $K'$ of $K_{a,b,\eta,\mu}$. For $\eta,\mu$ less than $1$, the only orbits of $K'$ are its critical points. Using the constant cappings, the maximum has action approximately $a\eta$, the minimum has action approximately $-b\mu$, and all others critical points have actions close to zero. The Conley-Zehnder indices of these orbits are related to their Morse indices in such a way that any non-constant recapping will not contribute to the spectral invariant of the unit, and in this fashion one concludes that $\gamma(K')\approx a\eta+b\mu;$ taking the limit $\eta,\mu\to 1$ yields the desired result.
\end{proof}

\subsection{Duality and inversion in rational symplectic manifolds}
\label{sec:duality-inversion-1}

We suppose throughout this section that $(W,\omega)$ is \emph{rational}, i.e., $\omega(\pi_{2}(M))$ is a discrete subgroup of $\R$; the minimal positive generator of this subgroup is denoted by $\rho$ and is called the \emph{rationality constant} of $(W,\omega)$.

The goal is to relate the spectral invariants of $\psi_{t}^{-1}$ with the spectral invariants of $\psi_{t}$. For related discussion we refer the reader to \cite{entov-poltero-IMRN-2003,usher-duality,leclercq-zapolsky}. We will show:
\begin{lemma}\label{lem:duality_lemma}
  Let $(W,\omega)$ be rational, and let $\psi_{t}$ be a contact-at-infinity Hamiltonian system with non-degenerate fixed points. For any class $b\in \mathrm{HF}(\psi_{t}^{-1})$, we have:
  \begin{equation*}
    c(\psi_{t}^{-1};b)=-\inf\set{c(\psi_{t};a):\ip{a,b}=1};
  \end{equation*}
  where $\ip{-,-}$ denotes the duality pairing $\mathrm{HF}(\psi_{t})\otimes \mathrm{HF}(\psi_{t}^{-1})\to \Z/2$ defined on generators by $\ip{(y,v),(x,u)}=1$ if and only if $x=y$ and $u=\bar{v}$; see \S\ref{sec:duality-isomorphism} for more details.
\end{lemma}
\begin{proof}
  The proof uses the duality isomorphism in \S\ref{sec:duality-isomorphism} for rational symplectic manifolds. For the rest of the argument we refer the reader to the proof of \cite[Lemma 2.2]{entov-poltero-IMRN-2003}.
\end{proof}

\subsubsection{The duality isomorphism}
\label{sec:duality-isomorphism}

Let $(y,v)$ be a capped orbit of $\psi_{t}$, and consider the capped orbit of $\psi_{t}^{-1}$ given by $(y,\bar{v})$ given by $\bar{v}(s,t)=\psi_{t}^{-1}(v(-s,t))$ as in \S\ref{sec:cappings}. Let us denote by:
\begin{equation*}
  \ip{-,-}:\mathrm{CF}(\psi_{t})\otimes \mathrm{CF}(\psi_{t}^{-1})\to \Z/2,
\end{equation*}
the pairing defined by:
\begin{equation*}
  \ip{(y,v),(x,u)}=\left\{
    \begin{aligned}
      &1&&x=y\text{ and }u=\bar{v},\\
      &0&&\text{otherwise,}
    \end{aligned}
  \right.
\end{equation*}
where the equality between $u$ and $\bar{v}$ holds in the space of symplectic cappings, as described in \S\ref{sec:cappings}.

We first observe that, if $\sum (y_{j},v_{j})$ and $\sum (x_{i},u_{i})$ are semi-infinite sums of the correct type to define Floer chains, then only finitely many terms in each can contribute to their pairing, because:
\begin{equation*}
  \ip{(y_{j},v_{j}),(x_{i},u_{i})}\ne 0\text{ if and only if }\mathscr{A}(\psi_{t};y_{j},v_{j})+\mathscr{A}(\psi_{t}^{-1};x_{i},u_{i})=0,
\end{equation*}
and only finitely many terms in each sum are equal to the negative of a term in the other sum.

It follows from the identification of \S\ref{sec:inverse-identification} that this pairing satisfies:
\begin{equation*}
  \ip{d (y,v),(x,u)}+\ip{(y,v),d (x,u)}=0,
\end{equation*}
and therefore the map:
\begin{equation}\label{eq:duality-map}
  \alpha\in \mathrm{CF}_{\le -A}(\psi_{t})\mapsto \ip{\alpha,-}\in \mathrm{Hom}(\mathrm{CF}_{\ge A}(\psi_{t}^{-1}),\Z/2)
\end{equation}
is a chain map, where $\mathrm{CF}_{\ge A}=\mathrm{CF}/\mathrm{CF}_{<A}$. When $(W,\omega)$ is rational it is not hard to see that \eqref{eq:duality-map} is an isomorphism on chain level; briefly, the reason is that any element of $\mathrm{Hom}(\mathrm{CF}_{\ge A}(\psi_{t}^{-1}),\Z/2)$ can be regarded as an infinite sum of the form $\sum_{j} \ip{(y_{j},v_{j}),-}$ with $\mathscr{A}(\psi_{t};y_{j},v_{j})\le -A$ (any such infinite sum is well-defined as an element of the dual space). The rationality assumption ensures that there are only finitely many capped orbits of $\psi_{t}$ with action in a given compact interval, and hence $\sum (y_{j},v_{j})$ is well-defined as an element of $\mathrm{CF}_{\le -A}(\psi_{t})$. This proves that \eqref{eq:duality-map} is surjective; the proof of injectivity is easier and is left to the reader.

\subsection{Naturality transformations}
\label{sec:natur-transf}

Given a contact-at-infinity system $\phi_{t}$ with $\phi_{1}=\id$ whose orbits $\phi_{t}(x)$ are contractible, one can associate a chain level \emph{naturality transformation} $$\mathfrak{n}:\mathrm{CF}(\psi_{t})\to \mathrm{CF}(\phi_{t}\psi_{t})$$ for any other contact-at-infinity system $\psi_{t}$. There is a close relationship between naturality transformations and the \emph{Seidel representation} of \cite{seidel-representation-GAFA-1997}.

Bearing in mind that $W$ is assumed to be connected, the operation depends on an auxiliary choice of capping of one of the orbits $\phi_{t}(x)$, namely, a cylinder $u:[-1,1]\times \R/\Z\to W$ so that $u(-1,t)=x$ and $u(1,t)=\phi_{t}(x)$. For any other choice of point $y\in W$, one can take any path $\eta(s)$ from $x$ to $y$ and define $u_{y}$ to be the capping of $y$ relative $\phi_{t}$ obtained by concatenating:
\begin{enumerate}
\item the inverse of $\eta$ (from $y$ to $x$),
\item the capping $u$ of $x$,
\item the cylinder $\phi_{t}(\eta)$ from $\phi_{t}(x)$ to $\phi_{t}(y)$.
\end{enumerate}
Since $\phi_{t}$ has zero flux, $u_{y}$ is independent of $\eta$ up to equivalence.

The transformation $\mathfrak{n}$ sends the capped orbit $(y,v)$ in $\mathrm{CF}(\psi_{t})$ to the concatenation $(y,u_{y}\# \phi_{t}(v))$. It is immediate from our definition of the Floer homology differential in terms of the time-$1$ map that $\mathfrak{n}$ is a chain map; see, e.g., \cite[pp.\,3308]{kislev_shelukhin} for a similar construction.

One computes that $\mathfrak{n}$ shifts action values according to the action of the capped orbit $(y,u_{y})$ of the system $\phi_{t}$. Since $(y,u_{y})$ is a critical point for the action functional (on the covering space of cappings), this action shift is independent of $y$. As in \cite[Proposition 31]{kislev_shelukhin}, it follows that:
\begin{lemma}\label{lem:nat}
  If $a\in \mathrm{HF}(\psi_{t})$, then $c(\phi_{t}\psi_{t};\mathfrak{n}(a))=c(\psi_{t};a)+\mathscr{A}(\phi_{t};x,u).$\hfill$\square$
\end{lemma}

\subsubsection{A particular naturality transformation}
\label{sec:part-natur-transf}

In this section we consider the particular loop $\phi_{t}$ on $W=\C^{n}$ generated by $H=\pi \abs{z}^{2}$. This Hamiltonian function is normalized and it generates the $\R/\Z$-action $\phi_{t}(z)=e^{2\pi it}z$. Therefore $\phi_{t}$ induces a naturality transformation on Floer homology:
\begin{equation*}
  \mathfrak{n}:\mathrm{HF}(\psi_{t})\to \mathrm{HF}(\phi_{t}\psi_{t}).
\end{equation*}
Since $\C^{n}$ is symplectically aspherical, there is a unique choice of auxiliary capping $(x,u)$; let us therefore take $x=0$ and $u$ to be the constant capping. It follows that $\mathfrak{n}$ is action preserving. To be precise:
\begin{lemma}
  Let $a\in \mathrm{HF}(\psi_{t})$ where $\psi_{t}$ is as above. Then the spectral invariant of $a$ equals the spectral invariant of $\mathfrak{n}(a)\in \mathrm{HF}(\phi_{t}\psi_{t})$.\hfill $\square$
\end{lemma}

Consider $\mathrm{HF}(\psi_{t}R_{-\epsilon t})$ for $\epsilon\in (0,1)$ and where $\psi_{t}$ is compactly supported. It is well-known that this group is $1$-dimensional over $\Z/2$. Indeed, the dimension is independent of the choice of $\psi_{t}$ and one can find a representative so that $\psi_{t}R_{-\epsilon t}$ has a single non-degenerate orbit located at $z=0$. It must therefore hold that the unit element is the sole non-zero element.

Let us then define $\mathrm{pt}\in \mathrm{HF}(\phi_{t}\psi_{t}R_{-\epsilon t})$ to be the image of $1$ under $\mathfrak{n}$. We call this element the ``point class'' since $\phi_{t}\psi_{t}R_{-\epsilon t}$ has a positive slope at infinity (namely $1-\epsilon$), and so the generator should be thought of as the minimum of a Morse function (rather than the unit element which should be thought of as the maximum of a Morse function). Since $\mathfrak{n}$ is an isomorphism, the point class generates $\mathrm{HF}(\phi_{t}\psi_{t}R_{-\epsilon t})$. This point class plays a role in \S\ref{sec:spectr-diam-ball}.

Using continuation or naturality maps, we therefore define $\mathrm{pt}\in \mathrm{HF}(\psi_{t}R_{st})$ for any positive slope $s\in (0,1)$ if $\psi_{t}$ is compactly supported. This element is natural, i.e., is preserved under continuation maps, because continuation maps commute with naturality maps; see, e.g., \cite[\S2.2.5]{cant-hedicke-kilgore}.

\subsection{The spectral diameter of a ball}
\label{sec:spectr-diam-ball}

This section proves that the spectral diameter of $B(a)$ equals $a$. The lower bound is well-known and follows from a simplified version of the construction in \S\ref{sec:spectr-diam-pack}. It is then enough to prove that $\gamma(\psi_{t}) \leq a$ for all systems $\psi_{t}$ compactly supported in $B(a)$. Without loss of generality, let us set $a=1$, and fix a system $\psi_{t}$ supported in $B(1)$.

The proof splits into five claims. The argument involves the naturality transformation $\mathfrak{n}$ from \S\ref{sec:part-natur-transf} which sends the unit class $1$ to the class $\mathrm{pt}$. It is important that $1$ generates $\mathrm{HF}(R_{-st}\psi_{t})$ and $\mathrm{pt}$ generates $\mathrm{HF}(R_{st}\psi_{t})$ for any $s\in (0,1)$ and any compactly supported $\psi_{t}$; i.e., there are two relevant Floer homologies and each is one-dimensional.

\begin{claim}\label{clm:five_claims}
Let $s\in (0,1)$ and for $\delta>0$ let $R_{s}=R_{\delta,s}$ be as in \S\ref{sec:reeb-flows}; then:
\begin{enumerate}[label=(\roman*)]
\item\label{item:h1} $\gamma(\psi_{t}) = 2s(1 + \delta/2) + c(R_{-st}\psi_{t}; 1) + c(R_{-st}\psi_{t}^{-1}; 1),$
\item\label{item:h2} $c(R_{-st}\psi_{t}^{-1}; 1) = - c(\psi_{t}R_{st}; \mathrm{pt}).$
\end{enumerate}
Let $\phi_{t}$ be the loop generated by $H = \pi |z|^{2}$, and suppose $s>1/2$; then:
\begin{enumerate}[resume,label=(\roman*)]
\item\label{item:h3} $c(R_{st}\psi_{t}; \mathrm{pt}) + c(\phi_{t} R_{-2st}; 1) \geq c(\phi_{t} R_{-st} \psi_{t}; \mathrm{pt}),$
\item\label{item:h4} $c(\phi_{t}R_{-st}\psi_{t}; \mathrm{pt}) = c(R_{-st}\psi_{t}, 1),$ and
\item\label{item:h5} $c(\phi_{t}R_{-2st};1)\le (1-2s)(1+\delta/2)$.
\end{enumerate}
\end{claim}

Before we prove the claim, we use it to determine the spectral diameter of the ball $B(1)$. We estimate:
\begin{equation*}
  \begin{aligned}
    \gamma(\psi_{t})
    &= 2s(1 + \delta/2) + c(R_{-st}\psi_{t}; 1) + c(R_{-st}\psi_{t}^{-1}; 1)\\
    &=2s(1 + \delta/2) + c(R_{-st}\psi_{t}; 1) - c(\psi_{t}R_{st}; \mathrm{pt}) \\
    &\leq 2s(1 + \delta/2) - c(\phi_{t} R_{-st} \psi_{t}; \mathrm{pt}) + c(\phi_{t} R_{-2st}; 1) + c(R_{-st}\psi_{t}; 1)\\
    &=2s(1 + \delta/2) + c(\phi_{t} R_{-2st}; 1)\\
    &\le 2s(1 + \delta/2) + (1-2s)(1+\delta/2) = 1 + \delta/2.
  \end{aligned}
\end{equation*}
The $j$th (in)equality uses the $j$th item in Claim \ref{clm:five_claims}. Since $\delta$ can be chosen arbitrarily small, we conclude $\gamma(\psi_{t}) \leq 1$. We now prove the five claims.

\begin{proof}[Proof of Claim \ref{clm:five_claims}]
  
We begin by recalling the definition of the spectral norm:
$$\gamma(\psi_{t}):= \lim_{\epsilon \to 0} c(R_{-\epsilon t} \psi_{t}; 1)+c(R_{-\epsilon t} \psi_{t}^{-1}; 1).$$
Therefore to prove \ref{item:h1} it is enough to prove the following identity for all systems $\psi_{t}$:
\begin{equation}\label{eq:s-epsilon}
  c(R_{-st} \psi_{t}; 1) +(s-\epsilon)(1+ \delta/2) =c(R_{-\epsilon t} \psi_{t}; 1).
\end{equation}
Let $K_{t}$ generate $\psi_{t}$ so that $R_{-\sigma t}\psi_{t}$ is generated by:
\begin{equation}
  \label{eq:K-sigma}
  K^{\sigma}_{t} = K_{t} - \sigma (\mu_{\delta}(\pi |z|^{2} -1) + 1),
\end{equation}
where we use the fact that $\psi_{t}$ and $R_{s}$ commute with each other. The cut-off function $\mu_{\delta}$ is defined in \S\ref{sec:reeb-flows}.

Observe that the 1-periodic orbits of the system $R_{-\sigma t}\psi_{t}$ are independent of $\sigma\in (0,1)$, and all remain in the ball $B(1)$; indeed, outside of the ball we have $R_{-\sigma t}\psi_{t}=R_{-\sigma t}$ which rotates with a non-zero speed because of our construction in \S\ref{sec:reeb-flows}. The action of the orbits depends on $\sigma$, as is clear from equation\eqref{eq:K-sigma}. Therefore a continuity argument (and the nowhere density of the action spectrum) implies that the spectral invariant must change according to:
\begin{equation*}
  \pd{}{\sigma}c(R_{-\sigma t}\psi_{t};1)=-(1+\mu_{\delta}(\pi \abs{z}^{2}-1))=-(1+\delta/2),
\end{equation*}
This implies equation \eqref{eq:s-epsilon}, and completes the proof of \ref{item:h1}.

The second item \ref{item:h2} follows from the duality formula proved in Lemma \ref{lem:duality_lemma} and the fact that the only non-zero class in $\mathrm{HF}(\psi_{t}R_{st})$ is the point class $\mathrm{pt}$.

The third item \ref{item:h3} is an immediate consequence of the subadditivity of the spectral invariants proved in \S\ref{sec:sub-addit-spectr}.

Item \ref{item:h4} follows from the fact that the naturality transformation defined by $\phi_{t}$ is action preserving and it sends the unit $1 \in \mathrm{HF}(R_{-st}\psi_{t})$ to the point class $\mathrm{pt}\in \mathrm{HF}(\phi_{t}R_{-st}\psi_{t})$ which is the only non-zero class in this group; see \S\ref{sec:natur-transf} for further discussion.

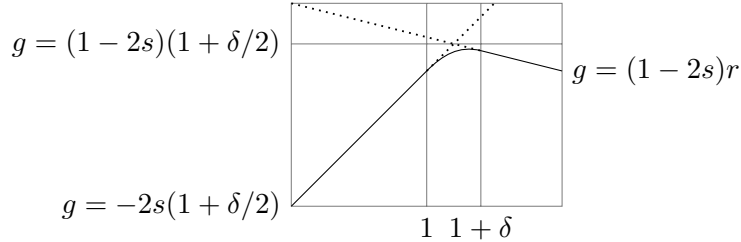
\begin{figure}[h]
  \centering
  \begin{tikzpicture}[scale=1.8]
    \draw[black!50!white] (0,-1.5) rectangle (2,0) (1.4,0)--(1.4,-1.5)node[below,black]{$1+\delta$} (1.0,0)--(1.0,-1.5)node[below,black]{$1\vphantom{1+\delta}$} (0,-0.3)node[left,black]{$g=(1-2s)(1+\delta/2)$}--(2,-0.3);
    \draw[dotted, line width=.7pt] (0,0)--(1.4,-0.35) (1,-.5)--(1.5,0);
    \draw (0,-1.5)node[left]{$g=-2s(1+\delta/2)$}--(1,-0.5)to[out=45,in=166](1.4,-.35)--(2,-0.5)node[right]{$g=(1-2s)r$};
  \end{tikzpicture}
  \caption{The graph of $g(r)$; its graph remains below the dashed lines $g=(1-2s)r$ and $g=r-2s(1+\delta/2)$ and coincides with one of these lines when $r\not\in (1,1+\delta)$.}
  \label{fig:graph-of-g}
\end{figure}

To establish \ref{item:h5} we argue as follows; first the Hamiltonian function $G$ generating $\phi_{t} R_{-2st}$ is equal to:
$$G = \pi |z|^{2} - 2s (\mu_{\delta}(\pi |z|^{2} - 1) + 1)=g(\pi\abs{z}^{2});$$
see Figure \ref{fig:graph-of-g} for the graph of $g(r)$.

The $g$-coordinate of the intersection of the two dashed lines in Figure \ref{fig:graph-of-g} equals $(1-2s)(1+\delta/2)$. It follows that every orbit of $\phi_{t}R_{-2st}$ has action bounded from above by this amount, assuming of course that $1/2<s<1$, and hence:
\begin{equation*}
  c(\phi_{t}R_{-2st};1)\le (1-2s)(1+\delta/2),
\end{equation*}
which is what we wanted to show.
\end{proof}

\subsection{Comparison with the bound of Entov and Polterovich}
\label{sec:comp-with-bound}

It is shown in \cite{entov-poltero-IMRN-2003} that the spectral diameter of $\mathbb{C}P^{n}$ is equal to the symplectic area of the class of a line, which we normalize to be 1. It also holds that $\mathbb{C}P^{n}$ contains a symplectically embedded copy of $B(1)$; let us fix the standard embedding $e:B(1)\to \mathbb{C}P^{n}$ whose image is the complement of the hyperplane divisor. Then, for any Hamiltonian system $\psi_{t}$ with compact support in $B(1)$, we can implant $\psi_{t}$ as a compactly supported system $e\psi_{t}e^{-1}$ in $\mathbb{C}P^{n}$, and know that $\gamma_{\mathbb{C}P^{n}}(e\psi_{t}e^{-1})\le 1$.

It is therefore natural to wonder whether $\gamma_{B(1)}(\psi_{t})\le \gamma_{\mathbb{C}P^{n}}(e\psi_{t}e^{-1})$ holds, as such a relation would imply our result from \S\ref{sec:spectr-diam-ball}. However, this strategy does not work because:

\begin{prop}\label{prop:only-dim-1}
  There exist systems $\psi_{t}$ supported in $B(1)$ so that: $$\gamma_{\mathbb{C}P^{1}}(e\psi_{t}e^{-1})<\gamma_{B(1)}(\psi_{t});$$
  the difference between the two sides can be made arbitrarily close to 1.
\end{prop}
We state the result only for the case $n=1$ as the argument is very simple in that case, although the authors expect a similar phenomenon holds in higher dimensions.\footnote{One approach to Proposition \ref{prop:only-dim-1} in higher dimensions is to exploit the fact that $\mathbb{C}P^{n-1}\subset \mathbb{C}P^{n}$ is \emph{stably displaceable} (see \cite[\S4.3]{gurel-CCM-2008}), and that the spectral invariants satisfy a K\"unneth formula (\cite[Theorem 5.1]{entov-polterovich-compositio-2009}); see \cite[\S3.6]{borman-JSG-2012} for a related discussion.}
\begin{proof}  
  Let $\psi_{t}$ be the Hamiltonian system generated by a radial bump function $H=f(\pi\abs{z}^{2})$ where $f$ is non-increasing, equals $1$ on $[0,1-\epsilon]$ and rapidly cuts off to zero so as to have compact support in $[0,1)$. A straightforward consideration of the spectrum shows that $\gamma_{B(1)}(\psi_{t})\ge 1-\epsilon$.

  On the other hand, by taking $\epsilon \to 0$, we arrange that $e\psi_{t}e^{-1}$ is supported in any chosen neighbourhood $U$ of the divisor $\mathbb{C}P^{n-1}$, which is a single point when $n=1$. The argument is finished by showing that:
  \begin{equation}\label{eq:spectral-diameter-symplectic}
    \inf\set{\gamma(U):\mathbb{C}P^{n-1}\subset U}=0;
  \end{equation}
  in words, the spectral diameter of small neighborhoods of the divisor is small. In the low-dimensional case $n=1$, \eqref{eq:spectral-diameter-symplectic} follows from the displacement energy bound of \S\ref{sec:displ-energy-bound}, completing the proof.
\end{proof}

\subsection{Spectral diameter of special balls in $\mathbb{C}P^{n}$}
\label{sec:spectr-diam-spec}

The proof of Theorem \ref{theorem:A} is based on the following two lemmas, in which we denote by the homology class in $M=\mathbb{C}P^{n}$ represented by the hyperplane by $\Gamma\in H_{2n-2}(M)$, the fundamental class by $[M]$ and the point class by $[\mathrm{pt}]$.
\begin{lemma}\label{lemma:A2}
  For any Hamiltonian system $\psi_{t}$ on $M=\mathbb{C}P^{n}$, we have:
  \begin{equation*}
	\gamma(\psi_{t}) + c(\varphi_{t}\psi_{t}; [M])-c(\varphi_{t}\psi_{t}; \Gamma)=1;
  \end{equation*}
  where $\varphi_{t}$ is the loop generated by $S_{a}$.
\end{lemma}

Introduce $U=\set{S_{a}>0}$, i.e., $U$ is the interior of the standard special ball of capacity $1-a$. Then:
\begin{lemma}\label{lemma:A3}
  For any Hamiltonian system $\psi_{t}$ supported in $U$, we have:
  \begin{equation*}
    c(\varphi_{t} \psi_{t}; [M])-c(\varphi_{t}\psi_{t}; \Gamma)\ge a.
  \end{equation*}
  where $\varphi_{t}$ is the loop generated by $S_{a}$.  
\end{lemma}
From the above lemmas, it follows that the spectral diameter of $U$ is at most $1-a$.  Combining this upper bound with the lower bound furnished by Proposition \ref{prop:minimal-chern} yields Theorem \ref{theorem:A}.

In the proofs of both lemmas we allow ourselves to work with non-normalized Hamilltonians on $\mathbb{C}P^{n}$, using the conventions in \S\ref{sec:non-norm-hamilt}. 

\subsubsection{Proof of Lemma \ref{lemma:A2}}
\label{sec:proof-lemma-A2}

The proof is a simple application of the naturality transformation associated to the loop generated by $S_{a}$ and the inverse loop generated by $\bar{S}_{a}$.

The naturality transformation associated to $S_{a}$ sends the class represented by the maximum of a perfect Morse function to the class of the index $2n-2$ critical point. Consequently by Lemma \ref{lem:nat}: 
\begin{equation*}
  c(H_{t}; [M])=c(S_{a}\# H_{t}; \Gamma)+\mathrm{const},
\end{equation*}
for every Hamiltonian $H_{t}$ supported in $U$. The constant term is the action of a capped orbit of $S_{a}$ and can be determined by sending $H_{t}\to 0$, in which case it must equal $c(0;[M])-c(S_{a};\Gamma)$. Since $S_{a}$ can be perturbed to a perfect Morse function whose index $2n-2$ critical point has action $-a$ (see Figure~\ref{fig:schematic-Sa}), and the Hessian at this critical point is small enough, it follows that $c(0;[M])-c(S_{a};\Gamma)=a$, and hence the above equation becomes:
\begin{equation}\label{eq:A2-1}
  c(H_{t};[M])=c(S_{a}\# H_{t}; \Gamma)+a.
\end{equation}
The next stage of the argument is similar but is based instead on the naturality transformation generated by $\bar{S}_{a}=-S_{a}$. This naturality transformation sends the class of the maximum to the class of the minimum, and hence:
\begin{equation*}
  c(S_{a}\# H_{t}; [M])=c(\bar{S}_{a}\# S_{a}\# H_{t}; [\mathrm{pt}])+\mathrm{const}=c(H_{t}; [\mathrm{pt}])+\mathrm{const}.
\end{equation*}
The constant term can again be determined by sending $H_{t}\to 0$, in which case it equals $c(S_{a};[M])$ which equals $1-a$, i.e., the critical value of the maximum as shown in Figure \ref{fig:schematic-Sa}. Thus:
\begin{equation}\label{eq:A2-2}
  c(S_{a}\# H_{t}; [M])=c(H_{t}; [\mathrm{pt}])+1-a.
\end{equation}
To complete the proof add together \eqref{eq:A2-1} and \eqref{eq:A2-2}, for $H_{t}$ generating $\psi_{t}$. \hfill $\square$

\subsubsection{Proof of Lemma \ref{lemma:A3}}
\label{sec:proof-lemma-A3}

Let $U=\set{S_{a}>0}$, as in the statement, and suppose that $H_{t}$ is a Hamiltonian function supported in $U$ generating $\psi_{t}$.

For the purposes of the proof, we introduce two time-reparametrization operations on a Hamiltonian function $G_{t}$ which does not affect its time-$1$ map in the universal cover:
\begin{equation*}
  G_{t}^*:=\beta'(2t)G_{\beta(2t)}\text{ and }G_{t}^{**}:=\beta'(2t-1)G_{\beta(2t-1)}.
\end{equation*}
Here $\beta:\R\to [0,1]$ is a standard smooth cut-off so $\beta(t)=0$ for $t\le 0$ and $\beta(t)=1$ for $t\ge 1$; we require that $\beta'(t)$ is non-negative.

The significance of these operations is the following: $G_{t}^*$ is supported where $t\in (0,1/2)$ while $G_{t}^{**}$ is supported where $t\in (1/2,1)$.

Due to the fact that spectral invariants depend only on the time-1 map in the universal cover (and the average value of the Hamiltonian), it follows that for any homology class $\Pi$, we have:
\begin{equation*}
  c(S_{a}\# H_{t}; \Pi)=c(S_{a}^{*}\# H_{t}^{**}; \Pi).
\end{equation*}
Now introduce the piecewise smooth function $K_{a}=\min\set{S_{a},0}$, as shown in Figure \ref{fig:proof-A3}. Since $K_{a}$ is pointwise less than $S_{a}$, it follows by a standard continuation argument that, for any class $\Pi$, we have
\begin{equation*}
  c(K_{a}^{*}\# H_{t}^{**}; \Pi)\le c(S_{a}^{*}\# H_{t}^{**}; \Pi).
\end{equation*}
We note that $K_{a}^{*}$ is not smooth, but nonetheless the spectral invariant $c(K_{a}^{*}\# H_{t}^{**};\Pi)$ is well-defined via a limiting process, because of the Hofer continuity of spectral invariants.

The next stage of the argument is a deformation argument. Roughly speaking, the idea is to interpolate from $K_{a}$ to $0$ while keeping track of the indices and actions of the orbits during the process. To make this precise, we argue in a slightly ad hoc fashion.

 Fix $0<\epsilon<a$ and introduce the $s$-dependent family of functions:
\begin{equation*}
  T_{s}=\max\set{K_{a},sK_{a}-\epsilon},
\end{equation*}
Note that on the set $\set{K_{a}\ge -\epsilon}=\set{S_{a}\ge -\epsilon}$, $T_{s}=T_{1}$. First smooth $T_{1}$ on this set so as to make $T_{1}^{*}\# H_{t}^{**}$ have non-degenerate orbits on $\set{K_{a}\ge -\epsilon/2}$.

Then, for each $s$, smooth $T_{s}$ on $\set{K_{a}\le -\epsilon/2}$ so that the only orbits outside in $\set{K_{a}\le -\epsilon/2}$ are the Morse critical points of indices $\set{0,2,\dots,2n-2}$ located near the divisor $\set{K_{a}=-a}$.

\begin{enumerate}
\item\label{item:Ts1} orbits contained entirely in the region where $T_{s}^{*}\# H_{t}^{**}=T_{1}^{*}\# H_{t}^{**}$,
\item\label{item:Ts2} orbits whose Floer homology grading is in $$\set{0,2,\dots,2n-2},$$ located near the divisor $\mathbb{C}P^{n-1}$.
\end{enumerate}

Because the minimal Chern number of $\mathbb{C}P^{n}$ is $n+1$, it follows that the orbits of type \eqref{item:Ts2} never appear in a linear combination representing $[M]$, which has degree $2n$. Consequently $c(T_{s}^{*}\# H_{t}^{**}; [M])$ is independent of $s$, since it is valued in the $s$-independent nowhere dense spectrum of orbits of type \eqref{item:Ts1}. Note that $T_{0}^{*}\# H_{t}^{**}$ is $\epsilon$-close to $H_{t}^{**}$, and thus we conclude that:
\begin{equation*}
  c(K_{a}^{*}\# H_{t}^{**};[M])=c(T_{1}^{*}\# H_{t}^{**};[M])=c(H_{t};[M])+O(\epsilon).
\end{equation*}
Since $\epsilon$ was arbitrary we conclude $c(K_{a}^{*}\# H_{t}^{**};[M])=c(H_{t};[M])$.

Combining the two steps, with $\Pi=[M]$, we conclude:
\begin{equation}\label{eq:A3}
c(S_{a}\# H_{t}; [M]) = c(S_{a}^{*}\# H_{t}^{**}; [M]) \ge c(K_{a}^{*}\# H_{t}^{**}; [M]) = c(H_{t}; [M]).
\end{equation}
Next, we recall \eqref{eq:A2-1}, so that:
\begin{equation*}
  c(S_{a}\# H_{t}; [M])-c(S_{a}\# H_{t}; \Gamma)=c(S_{a}\# H_{t}; [M])-c(H_{t}; [M])+a\ge a,
\end{equation*}
where we use \eqref{eq:A3} in the final step. This completes the proof.\hfill$\square$

\begin{figure}[h]
  \centering
  \begin{tikzpicture}[scale=1.9]
    \draw (0,0)--(1,0) (0,-0.3) --+(0,1);
    \draw[line width=0.6pt] (0,-0.3)--(1,0.7);
    \draw[->] (1.4,0.2)--node[above]{Step 1}+(0.5,0);
    \node at (0.5,-0.5){$S_{a}$};
    \begin{scope}[shift={(2.3,0)}]
      \draw (0,0)--(1,0);
      \draw (0,-0.3) --+(0,1);
      \draw[dotted,line width=0.7pt] (0.3,0)--(1,0.3) (0.3,0)--(1,0.7);
      \draw[line width=0.6pt] (0,-0.3)--(0.3,0)--(1,0);
      \draw[->] (1.4,0.2)--node[above]{Step 2}+(0.5,0);
      \node at (0.5,-0.5){$K_{a}$};
    \end{scope}
    \begin{scope}[shift={(4.6,0)}]
      \draw (0,0)--(1,0);
      
      \draw (0,-0.3) --+(0,1);
      \draw[line width=0.6pt] (0,-0.05)--(0.25,-0.05)--(0.3,0)--(1,0);
      \node at (0.5,-0.5){$\max\{-\epsilon,K_{a}\}$};
    \end{scope}
  \end{tikzpicture}
  \caption{Deformation used in the proof of Lemma \ref{lemma:A3}.}
  \label{fig:proof-A3}
\end{figure}
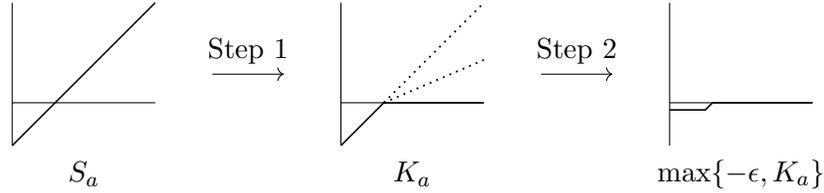

\bibliographystyle{alpha}
\bibliography{citations}
\end{document}